\documentclass[10pt]{article}
\usepackage{mathrsfs}
\usepackage{amsmath,amssymb}
\usepackage{amsthm}
\usepackage[colorlinks,linkcolor=red,anchorcolor=blue,citecolor=green]{hyperref}
\usepackage{graphics,graphicx}
\usepackage{color}
\usepackage{enumerate}
\usepackage{appendix}
\allowdisplaybreaks

\usepackage{tikz}

\newtheorem{thm}{Theorem}[section]
\newtheorem{lem}[thm]{Lemma}

\newtheorem{remark}[thm]{Remark}

\def\qed{\hfill \rule{4pt}{7pt}}

\numberwithin{equation}{section}

\def\R{{\mathbb{R}}}

\def\P{{\mathbb{P}}}

\def\Z{{\mathbb{Z}}}
\def\N{{\mathbb{N}}}

\def\F{{\mathfrak{F}}}

\def\sF{{\mathcal{F}}}

\def\sP{{\mathcal{P}}}

\def\sS{{\mathcal{S}}}

\def\sX{{\mathcal{X}}}

\def\bfe{{\mathbf{e}}}

\def\FSF{{\mathsf{FSF}}}
\def\WSF{{\mathsf{WSF}}}
\def\USF{{\mathsf{USF}}}
\def\UST{{\mathsf{UST}}}
\def\RW{{\mathrm{RW}}}

\usepackage{comment}
\usepackage[normalem]{ulem}
\definecolor{lw}{RGB}{0,0,255}

\begin{document}
\begin{center}
{{\large\bf Uniform spanning forests associated with biased random walks\\ on Euclidean lattices}\footnote{The project is supported partially by CNNSF (No.~11671216).}}
\end{center}
\vskip 2mm
\begin{center}
Z.
 Shi, V.
  Sidoravicius, H.
  Song, L. Wang,  K. Xiang
\end{center}
\vskip 2mm


\begin{abstract}

    The uniform spanning forest measure ($\USF$) on a locally finite, infinite connected graph $G$ with {conductance} $c$ is defined as a 
    weak limit of uniform spanning tree measure on finite subgraphs.  Depending 
    on the underlying graph and conductances, the corresponding $\USF$  is not necessarily concentrated on the set of spanning trees. 
    Pemantle~\cite{PR1991} showed that  on  $\Z^d$, equipped with the {\it unit }conductance $ c=1$,   $\USF$ is concentrated on spanning trees if and only if $d \leq 4$. In this work we study the $\USF$ associated with conductances induced by $\lambda$--biased random walk on $\Z^d$, $d \geq 2$, $0 < \lambda < 1$, {\it i.e.} conductances are set to be $c(e) = \lambda^{-|e|}$,  where $|e|$ is the graph distance of the edge $e$ from the origin. Our main result states that in this case $\USF$ consists 
    of finitely many trees if and only if $d = 2$ or $3$. More precisely, we prove that the uniform spanning forest has $2^d$ trees if $d = 2$ or $3$, 
    and infinitely many trees if $d \geq 4$. Our method relies on the analysis of the spectral radius and the speed of the $\lambda$--biased random walk on $\Z^d$.

    \bigskip

\medskip


\medskip


\end{abstract}

\section{Introduction and main results}

Let $G:= (V(G), \, E(G))$ be a locally finite, connected infinite graph and fix a vertex $o$ in $G$ as root.
For $x \in V(G)$, let $|x|$ be the graph distance of $x$ from $o$. We define, for $n \ge 0$,
$$
B_G(n)
:=
\{x\in V(G):\ |x| \le n\},
\qquad
\partial B_G(n)
:=
\{x\in V(G):\ |x| =n\}.
$$

\noindent Let $\lambda>0$. The $\lambda$-biased random walk, or $\RW_\lambda$, is a random walk on $(G,\, o)$ with transition probabilities: for $y$ adjacent to $x$, 
\begin{eqnarray}
    \label{(1.1)}
  p_\lambda(x,y)=
  \begin{cases}
    \frac{1}{d_o} &{\rm if}\ x=o,\\
      \frac{\lambda}{d_x+\left(\lambda-1\right)d_x^-} &{\rm if} \ x\neq o \text{ and }  |y| = |x|-1, \\
      \frac{1}{d_x+\left(\lambda-1\right)d_x^-} & \text{otherwise.}
  \end{cases}
\end{eqnarray}

\noindent Here $d_x$ is the degree of vertex $x$, and $d_x^-$ (resp. $d_x^0$) is the number of edges connecting $x$ to $\partial B_G(|x|-1)$ (resp. $\partial B_G(|x|)$). Note that $d_x^{-}\ge  1$ if $x\not=o$, and $d_o^-=d_o^0=0$. When $\lambda = 1$, $\RW_\lambda$ is the usual simple random walk on $G$. For general properties of biased random walks on graphs we refer to \cite{LR-PY2016} and \cite{SSSWX2017a+}.

In this work we study the uniform spanning forest on the network associated with $\RW_\lambda$. It relies on the analysis of the spectral radius and the speed of the walk. More specifically, we focus on the spectral radius and the speed of $\lambda$--biased random walk on the $d$-dimensional lattice $\Z^d$, and always assume $0 < \lambda < 1$, unless it is stated otherwise.

From \eqref{(1.1)} one can see that $\RW_\lambda$ on $\Z^d$ is closely related to the drifted random walk on $\Z^d$, whose distribution is given by convolutions of step distribution
\begin{equation}
    \label{e:mu}
\mu(\bfe_1) = \cdots = \mu(\bfe_d) =
    \frac{1}{d(1+\lambda)}, \quad \mu(-\bfe_1) = \cdots =
    \mu(-\bfe_d) = \frac{\lambda}{d(1+\lambda)},
\end{equation}

\noindent where $\{\bfe_1, \ldots, \bfe_d\}$ is the standard basis of $\Z^d$. Before exiting from one of the $2^d$ open orthants, the $\lambda$--biased random walk and drifted random walk have the same distributions. This fact is crucial for the analysis of spectral radius, speed and intersection properties of $\lambda$--biased random walks on $\Z^d$.
However,  $\lambda$--biased random walks exhibit quite different behavior from drifted random walk when they hit some axial hyperplane or the boundary of the orthant.

Let $(X_n)$ be the $\RW_\lambda$ on $\Z^d$ and $p^{(n)}_\lambda (x,\, y) = \P_x (X_n = y)$ be the $n$-step transition probability of $X_n$,
where $\mathbb{P}_x$ is the law of $\RW_\lambda$ starting at $x$.
The spectral radius $\rho(\lambda)$ of $\RW_\lambda$ is defined to be the reciprocal of the convergence radius for the Green function
$$
\mathbb{G}_\lambda(x,y | z)
:=
\sum_{n = 0}^\infty p^{(n)}_\lambda(x,y) z^n.
$$

\noindent Clearly, $\rho(\lambda)$ does not depend on the choices of $x$ and $y$, and can be expressed as
$$
\rho(\lambda)
:=
\limsup_{n \to \infty} [p^{(n)}_\lambda (o,\, o)]^{1/n}.
$$

\noindent Define the speed $\mathcal{S}(\lambda)$ of $\RW_{\lambda}$ by
$$
\mathcal{S}(\lambda)
:=
\lim_{n \to \infty} \frac{|X_n|}{n},
$$

\noindent provided the limit exists almost surely. 
\medskip

 There are many deep and important questions related to how the spectral radius and the speed depend on the bias parameter $\lambda$.
 Lyons, Pemantle and Peres \cite{LR-PR-PY1996b} asked  whether the speed of $\RW_\lambda$ on the supercritical Galton--Watson tree without leaves is strictly decreasing. This has been confirmed for $\lambda$  lying in some regions    (cf.\ \cite{BG-FA-SV2014,AE2014,AE2013,SH-WL-XK2015}),   but still remains open for  general values of $\lambda$. For the supercritical Galton--Watson tree with leaves, the speed is expected (\cite[Section~3]{BG-FA2014}) to be unimodal in $\lambda$ (due to presence of traps). On lamplighter graph $\Z \ltimes \sum_{x \in \Z} \Z_2$, the speed of $\RW_\lambda$ is positive if and only if $1 < \lambda < (1+\sqrt{5})/2$; see \cite{LR-PR-PY1996a}.
 
\medskip

We are ready to state our first main result. Its proof is given in Section~\ref{s:spectral}.

\begin{thm}\label{T:main1}

 Let $\lambda \in (0, \, 1)$ and let $(X_n)$ be $\RW_\lambda$ on $\Z^d$.
\begin{enumerate}[(i)]
\item
The spectral radius is $\rho(\lambda) = \frac{2\sqrt{\lambda}}{1+\lambda}<1$. 

\item The speed exists, and equals $\sS(\lambda)=\frac{1-\lambda}{1+\lambda}$. 
\end{enumerate}

\end{thm}

It is straightforward from the expressions above that the spectral radius is strictly increasing  in $\lambda$ 
and speed is strictly decreasing in $\lambda$.

%
%


\medskip

We now turn to the main topic of the paper, namely, the study of the uniform spanning forest of the network associated with the $\RW_\lambda$, by applying Theorem~\ref{T:main1}. Viewing $G=(V(G), \, E(G))$ as an infinite network with appropriate conductances on its edges, the uniform spanning forest measures are defined as weak limit of uniform spanning tree measures of finite subgraphs of $G$. The limit can be taken with either free or wired boundary conditions, yielding the free uniform spanning forest measure (denoted by $\FSF$) and the wired uniform spanning forest measure ($\WSF$), respectively. In general, $\FSF$ stochastically dominates $\WSF$ on any infinite network. If they coincide, we call them the uniform spanning forests ($\USF$) for simplicity. For more details, see Section \ref{s:usf} (or \cite{LR-PY2016}).

Both $\FSF$ and $\WSF$ on an infinite network are concentrated on the set of spanning forests with the property that every tree (i.e., the connected component) in the forest is infinite. 
When $\lambda=1$, the remarkable result of Pemantle~\cite{PR1991}  (see also \cite{LR-PY2016}) states that $\USF$ on $\Z^d$ has a single tree for $d\le 4$ and has infinitely many trees for $d \ge 5$. By \cite[Theorem~9.4]{BI-LR-PY2001}, this type of phase transition depending on the  dimension has a deep connection with the well-known intersection property of independent simple random walks on $\Z^d$, namely, two independent simple random walks on $\Z^d$ intersect infinitely often if $d \le 4$ and finitely many times if $d \ge 5$; see for example Lawler \cite{LG1980,LG1991}. We show that there is a phase transition for the number of trees in the USF on the network associated with $\RW_\lambda$ on $\Z^d$ with $0<\lambda < 1$, while the critical dimension is reduced from $4$ to $3$.

\begin{thm}
\label{T:USF}

Let $\lambda \in (0, \, 1)$. Almost surely, the number of trees in the uniform spanning forest associated with $\RW_{\lambda}$ on $\Z^d$ is $2^d$ if $d = 2$ or $3$, 
and is infinite if $d \ge 4$.

\end{thm}

Theorem~\ref{T:USF} is (restated and proved) in Section~\ref{s:usf}.
As we mentioned before, an important step in the proof is to
determine the number of intersections of two independent random walks. We state the result below and its proof is given in Section~\ref{s:intersection}.

\begin{thm}
 \label{T:intersection}

 Assume $\lambda\in (0,\, 1)$. Let $(Z_n)_{n=0}^{\infty}$ and $(W_n)_{n=0}^{\infty}$ be independent drifted random walks on $\Z^d$ with the same step distribution $\mu$ given by \eqref{e:mu}, starting at $z_0$ and $w_0$ respectively. Then almost surely,
 $$
 | \{Z_m;\ m\ge 0\} \cap \{W_n;\ n\ge 0\} |
 \ \text{is finite for}\ d\ge 4\ \text{and infinite for}\ d\le 3.
 $$

\end{thm}

The rest of the paper is organised as follows. In Section~\ref{s:spectral}, we prove sharp estimates for the $n$-step transition probability and the strong law of large numbers of $\RW_\lambda$ on $\Z^d$. The statements for the spectral radius and the speed in Theorem~\ref{T:main1} are direct consequences. The number of intersections of two independent drifted (or biased) random
walks is studied in Section~\ref{s:intersection}. In Section~\ref{s:usf} we consider the uniform spanning forests associated with $\RW_\lambda$, and prove Theorem~\ref{T:USF}.


\medskip

\section{Spectral radius and speed}
\label{s:spectral}

In this section, we prove Theorem \ref{T:main1} for $\RW_\lambda$ on $\mathbb{Z}^d$. In fact, we obtain sharp estimates for the $n$-step transition probability (Theorem~\ref{T:hkZd}), and establish a strong law of large numbers (Theorem~\ref{T:speed}). Theorem~\ref{T:main1} is a straightforward consequence of Theorems~\ref{T:hkZd} and
\ref{T:speed}.

For positive functions $f$ and $g$ on $\N$, we write $f\asymp g$ if there is a constant $c>0$ such that $c^{-1} g(n) \le f(n) \le c g(n)$ for all $n \in \N$, and write $f \sim g$ if $\lim_{n\to\infty} \frac{f(n)}{g(n)} = 1$.

\begin{thm}
 \label{T:hkZd}

 Let $\lambda \in (0, \, 1)$, and let $(X_n)$ be $\RW_\lambda$ on $\Z^d$. Then
 \begin{equation}
     \label{e:hk}
     p^{(2n)}_\lambda(o,\, o)
     \asymp
     \Big(\frac{2\sqrt{\lambda}}{1+\lambda}\Big)^{\! 2n} \frac{1}{n^{3d/2}}.
 \end{equation}

 \noindent In particular, the spectral radius equals $\rho(\lambda) = \frac{2\sqrt{\lambda}}{1+\lambda}<1$, and is strictly increasing in $\lambda$.

\end{thm}





The proof of Theorem \ref{T:hkZd} relies on the following lemma, which is motivated by \cite[Exercise 1.7]{PG2015}. For $0 \le k < n$, let $B_{n,k}$ be the set of paths $(x_0, \ldots, x_{2n})$ taking values in $\Z$ with $x_0=0=x_{2n}$ and $\# \{ i: \ 1\le i\le 2n, \, x_i = 0\} = k+1$. Here and throughout, by a path we mean $|x_i-x_{i-1}|=1$ for all $i$ (in other words, it is a possible trace of a simple random walk), and $2n$ is called the length of the path.

\begin{lem}
 \label{Lem5.1}

 There is a positive constant $c$ such that $|B_{n,k}|\le \frac{c \, k^{5/2}4^n}{n^{3/2}}$ for $n\in\mathbb{N}$ and $k\in [0,\, n]$.

\end{lem}

\begin{proof}
The lemma holds if $k=0$ ($|B_{n,k}|=0$ in this case), or if $k \ge \frac{n}{2}$ (using the trivial inequality $|B_{n,k}| \le 2^{2n}$).

Assume now $0<k<\frac{n}{2}$. For $\ell\ge 1$, let $C_\ell= \frac{1}{\ell+1} {2\ell \choose \ell}$ be the $\ell$-th Catalan number. The number of paths $(x_0, \ldots, x_{2\ell})$ on $\mathbb{Z}$ with length $2\ell$ such that $x_0=x_{2\ell}=0\in \Z$ and that $x_i \not= 0$ for $1\le i\le 2\ell-1$ is $2C_{\ell-1}$; such paths are the so-called excursions. By splitting the paths in $B_{n,k}$ into excursions, we see that
$$
|B_{n,k}|
\le
\sum_{\substack{n_1 + \cdots + n_k = n \\ n_i \ge 1,\, 1\le i\le k}}
(2C_{n_1-1})(2C_{n_2-1})\cdots (2C_{n_k-1})
=
2^k \sum_{\substack{n_1 + \cdots + n_k = n \\ n_i \ge 1,\, 1\le i\le k}}
C_{n_1-1} C_{n_2-1} \cdots C_{n_k-1} \, .
$$

\noindent 
Since $n_i \geq \frac{n}{k}$ for some $i$, we have that
\begin{align*}
    |B_{n,k}|
 &\le k \, 2^k \sum_{\substack{n_1 + \cdots + n_k = n \\ n_i \ge 1,\, 2\le i\le k, \; n_1 \ge n/k}}
    C_{n_1-1} C_{n_2-1} \cdots C_{n_k-1}
    \\
 &= k \, 2^k \sum_{\substack{n_2 + \cdots + n_k \le n-(n/k)\\ n_i \ge 1,\, 2\le i\le k}}
    C_{n-n_2-\cdots-n_k-1} C_{n_2-1} \cdots C_{n_k-1} \, .
\end{align*}

\noindent Recall that (\cite{DR-BR1986}) $C_\ell < \frac{4^\ell}{(\ell+1)\, (\pi \ell)^{1/2}}$ for all $\ell$. So $C_{n-n_2-\cdots-n_k-1} < \frac{4^{n-n_2-\cdots-n_k-1}}{(n-n_2-\cdots-n_k)\, (\pi (n-n_2-\cdots-n_k-1))^{1/2}}$, which is bounded by $\frac{4^{n-n_2-\cdots-n_k-1}}{\frac{n}{k}\, (\pi (\frac{n}{k}-1))^{1/2}}$ if $n_2 + \cdots + n_k \le n - \frac{n}{k}$. Accordingly,
$$
|B_{n,k}|
\le
\frac{k \, 2^k \, 4^{n-1}}{\frac{n}{k}\, (\pi (\frac{n}{k}-1))^{1/2}}
\sum_{\substack{n_2 + \cdots + n_k \le n-(n/k) \\ n_i \ge 1,\, 2\le i\le k}} \frac{C_{n_2-1}}{4^{n_2}} \cdots \frac{C_{n_k-1}}{4^{n_k}}
\le
\frac{k \, 2^k \, 4^{n-1}}{\frac{n}{k}\, (\pi (\frac{n}{k}-1))^{1/2}} \Big( \sum_{m=1}^\infty \frac{C_{m-1}}{4^m} \Big)^{k-1} .
$$

\noindent Recall that the generating function of $C_\ell$ is
$$
\sum_{\ell=0}^\infty C_\ell x^\ell
=
\frac{1 - 2\sqrt{1-x}}{2 x},
\qquad
x \in \Big[ - \frac{1}{4}, \, \frac{1}{4} \Big],
$$

\noindent from which it follows that $\sum_{\ell=0}^\infty \frac{C_\ell}{4^{\ell+1}} = \frac12$. Hence
$$
|B_{n,k}|
\le
\frac{k \, 2^k \, 4^{n-1}}{\frac{n}{k}\, (\pi (\frac{n}{k}-1))^{1/2}} \, \frac{1}{2^{k-1}}
=
\frac{2k \, 4^{n-1}}{\frac{n}{k}\, (\pi (\frac{n}{k}-1))^{1/2}} .
$$

\noindent Since $0<k<\frac{n}{2}$, we have $(\pi (\frac{n}{k}-1))^{1/2} = (\frac{\pi n}{k})^{1/2} \, (1-\frac{k}{n})^{1/2} \ge (\frac{\pi n}{2k})^{1/2}$, so that $|B_{n,k}| \le \frac{2^{3/2}k^{5/2} \, 4^{n-1}}{\pi^{1/2}n^{3/2}}$ as desired.
\end{proof}


\begin{proof}[Proof of Theorem~\ref{T:hkZd}]

\textbf{Step 1.} We first show the lower bound for $p^{(2n)}_\lambda(o,\, o)$.

Let $n>d$. We get a lower bound for $p^{(2n)}_\lambda(o,\, o)$ by considering only the paths starting at $o$ that reach $(1, \ldots, 1)$ at step $d$ (which happens with probability greater than or equal to $(\frac{1}{d(1+\lambda)})^d$), then stay in the first open orthant $\{ x = (x_1, \ldots, x_d) \in \Z^d: \, x_i > 0, \; 1 \le i \le d\}$ for the next $2(n-d)$ steps and end up at $(1, \ldots, 1)$ again (of which we are going to estimate the probability), and finally return to $o$ at step $2n$ (which happens with probability greater than or equal to $(\frac{\lambda}{d(1+\lambda)})^d$). To compute the probability that, starting at $(1, \ldots, 1)$, the walk stays in the first open orthant for $2(n-d)$ steps and ends up at $(1, \ldots, 1)$, we observe, by decomposing the paths into excursions as in the proof of Lemma \ref{Lem5.1}, that the total number of possible such paths is at least
$$
\sum_{\substack{n_1 + \cdots + n_d = n - d \\ n_i \ge 0,\; 1\le i\le d}} {2n \choose 2 n_1,\cdots, 2n_d} C_{n_1} \cdots C_{n_d} ,
$$

\noindent where $C_\ell$ denotes as before the Catalan number, and ${2n \choose 2 n_1,\cdots, 2n_d} := \frac{(2n)!}{(2n_1)! \cdots (2n_d)!}$ is the multinomial coefficient. By definition, the transition probability that $\RW_{\lambda}$, along such paths, steps forward (resp.\ backward) along each coordinate in the first open
orthant is $\frac{1}{d(1+\lambda)}$ (resp.\ $\frac{\lambda}{d(1+\lambda)}$), with the number of both forward and backward steps being $n-d$. Consequently,
\begin{align*}
    p^{(2n)}_\lambda(o,\, o)
&\ge \left(\frac{1}{d(1+\lambda)}\right)^d  \left(\frac{\lambda}{d(1+\lambda)}\right)^d\\
&\ \ \ \
    \sum_{\substack{n_1 + \cdots + n_d = n - d \\ n_i \ge 0,\; 1\le i\le d}} {2n \choose 2 n_1,\cdots, 2n_d} C_{n_1} \cdots C_{n_d} \left(\frac{1}{d(1+\lambda)}\right)^{n-d} \left(\frac{\lambda}{d(1+\lambda)}\right)^{n-d}\\
 &= \frac{\lambda^n}{[d(1+\lambda)]^{2n}} \sum_{\substack{n_1 + \cdots + n_d = n - d \\ n_i \ge 0,\; 1\le i\le d}} {2n \choose 2 n_1,\cdots, 2n_d} C_{n_1} \cdots C_{n_d} .
\end{align*}

\noindent Since $C_\ell = \frac{1}{\ell+1} {2\ell \choose \ell}$, and $(n_1+1) \cdots (n_d +1) \le (n+1)^d$, we get
$$
p^{(2n)}_\lambda(o,\, o)
\ge
\frac{\lambda^n}{[d(1+\lambda)]^{2n}} \frac{1}{(n+1)^d} \sum_{\substack{n_1 + \cdots + n_d = n - d \\ n_i \ge 0,\; 1\le i\le d}} {2n \choose 2 n_1,\cdots, 2n_d} {2n_1 \choose n_1} \cdots {2n_d \choose n_d}
$$

\noindent Note that $\sum_{\substack{n_1 + \cdots + n_d = n - d \\ n_i \ge 0,\; 1\le i\le d}} {2n \choose 2 n_1,\cdots, 2n_d} {2n_1 \choose n_1} \cdots {2n_d \choose n_d}$ equals $(2d)^{2n-2d} \, q^{(2n - 2 d)}(o, \, o)$, where $q^{(2k)}(o,\, o)$ is the $(2k)$-step transition probability, from $o$ to $o$, of the (unbiased) simple random walk on $\Z^d$. Since $k^{d/2} q^{(2k)}(o,\, o)$ converges, as $k\to \infty$, to a strictly positive limit (\cite[Theorem~1.2.1]{LG1991}), it follows that for some constant $c_1>0$ (depending on $d$ and on $\lambda$) and all sufficiently large $n$,
$$
p^{(2n)}_\lambda(o,\, o)
\ge
c_1 \, \Big( \frac{2\lambda^{1/2}}{1+\lambda} \Big)^{\! 2n} \frac{1}{(n+1)^d\, n^{d/2} } ,
$$

\noindent yielding the desired lower bound for $p^{(2n)}_\lambda(o,\, o)$.

\textbf{Step 2.} It remains to prove the upper bound for $p^{(2n)}_\lambda(o,o)$.

Let $\sP_{2n}$ be the set of paths from $o$ to $o$ on $\Z^d$ with length $2n$. For $\gamma := o \, \omega_1 \, \omega_2 \cdots \omega_{2n-1} \, o \in \sP_{2n}$, let
$$
\P(\gamma, \, \lambda)
:=
p_{\lambda}(o,\, \omega_1) \, p_{\lambda}(\omega_1, \, \omega_2) \ldots p_{\lambda}(\omega_{2n-1}, \, o) ,
$$

\noindent which stands for the transition probability of $\RW_{\lambda}$ along $\gamma$. Define, for $1\le i\le d$, the projection $\phi_i$: $\Z^d \to \Z$ by $\phi_i(y) := y_i$ for $y:= (y_1,\ldots,y_d)\in \Z^d$. Let $\gamma_i$ be the path on $\mathbb{Z}$ obtained from $\phi_i(\gamma) := \phi_i(o)\, \phi_i(\omega_1)\cdots \phi_i(\omega_{2n-1})\, \phi_i(o)$ by deleting all null moves. Let $n(\gamma)$ and $n(\gamma_i)$ be respectively the numbers of hits (but excluding the initial hit) to the axial hyperplanes of $\gamma$ and $\gamma_i$; hence $n(\gamma)$ and $n(\gamma_i)$ are odds numbers, with $n(\gamma) \ge 2$ (due to the initial and ending positions), and
$$
n(\gamma) \ge n(\gamma_1) + \cdots + n(\gamma_d).
$$

Consider the first $2n$ steps of $\RW_{\lambda}$ along the path $\gamma$. Each time the walk is inside some open orthant, the transition probability for the next step is either $\frac{1}{d (1+\lambda)}$ or $\frac{\lambda}{d(1+\lambda)}$, whereas each time it hits an axial hyperplane (which happens $n(\gamma)$ times by definition), the transition probability is of the form $\frac{1}{d + k + (d-k)\lambda}$ (with $1\le k\le d$) or $\frac{\lambda}{d+k + (d-k)\lambda}$ (with $1\le k\le d-1$). Note that $d+k + (d-k)\lambda \ge d ( 1 +\lambda) + 1 - \lambda$. The total number of probability terms of the forms $\frac{1}{d + k + (d-k)\lambda}$ or $\frac{1}{d (1+\lambda)}$ is exactly $n$, so is the total number of probability terms of the forms $\frac{\lambda}{d + k + (d-k)\lambda}$ or $\frac{\lambda}{d (1+\lambda)}$. Therefore, writing $\eta := \frac{d(1 + \lambda)}{d (1 + \lambda) + 1 - \lambda} \in (0, \, 1)$, we get, for $\gamma \in \sP_{2n}$,
\begin{equation}
    \label{e:Pgammalambda}
    \P(\gamma, \, \lambda)
    \le
    \eta^{n(\gamma)} \left( \frac{1}{d (1+\lambda)} \right)^{\! n} \left( \frac{\lambda}{d (1+\lambda)} \right)^{\! n}
    \le
    \eta^{n(\gamma_1) + \cdots + n(\gamma_d)}
    \Big( \frac{\sqrt{\lambda}}{d(1 + \lambda)} \Big)^{\! 2n}.
\end{equation}

Let $\sP_{2n}^0 \subset \sP_{2n}$ be the set of paths $\gamma$ that is contained in the hyperplane $\{ (x_1, \ldots, x_d)\in \Z^d: \ x_i = 0\}$ for some $1 \le i \le d$. By definition, $n(\gamma) = 2n$ for $\gamma \in \sP_{2n}^0$. Since $\# \sP_{2n}^0 \le \# \sP_{2n} \le (2d)^{2n}$, we have
\begin{equation}
    \label{e:P2n0}
    \sum_{\gamma \in \sP_{2n}^0} \P(\gamma, \, \lambda)
    \le
     (2d)^{2n} \eta^{2n} \left( \frac{1}{d (1+\lambda)} \right)^{\! n} \left( \frac{\lambda}{d (1+\lambda)} \right)^{\! n}
     =
     \Big( \frac{2 \sqrt{\lambda}}{1+\lambda} \Big)^{\! 2n} \eta^{2n}.
\end{equation}

We now consider the case $\gamma \in \sP_{2n}\setminus\sP_{2n}^0$. By Lemma \ref{Lem5.1},
\begin{align*}
    \sum_{\gamma \in \sP_{2n}\setminus\sP_{2n}^0}
    \eta^{n(\gamma_1) + \cdots + n(\gamma_d)}
 &\le c \sum_{\substack{n_1 + \cdots + n_d = n \\ n_i \ge 1, \; 1\le i\le d}} {2n \choose 2n_1, \ldots, 2n_d}
    \prod_{j=1}^d \sum_{k_j = 1}^{n_j} \eta^{k_j}
    \frac{k_j^{5/2} 4^{n_j}}{n_j^{3/2}}
    \\
 &\le c \left(\sum_{k=1}^{\infty} \eta^k k^{5/2}\right)^{\! d}
    4^n \sum_{\substack{n_1 + \cdots + n_d = n \\ n_i \ge 1, \; 1\le i\le d}} {2n \choose 2n_1, \ldots, 2n_d}
    \prod_{j=1}^d n_j^{-3/2}\, .
\end{align*}

\noindent In view of \eqref{e:Pgammalambda}, we obtain, with $c_2 := c (\sum_{k=1}^{\infty} \eta^k k^{5/2})^d <\infty$,
\begin{equation}
    \sum_{\gamma \in \sP_{2n}\setminus\sP_{2n}^0}
    \P(\gamma, \, \lambda)
    \le
    c_2 \, \Big( \frac{2\sqrt{\lambda}}{d(1 + \lambda)} \Big)^{\! 2n}
    \sum_{\substack{n_1 + \cdots + n_d = n \\ n_i \ge 1, \; 1\le i\le d}} {2n \choose 2n_1, \ldots, 2n_d}
    \prod_{j=1}^d n_j^{-3/2}\, .
    \label{e:Pgammalambda2}
\end{equation}

\noindent To study the expression on the right-hand side, we consider (unbiased) simple random walk on $\Z^d$, and let $S_i$ be the number of steps among the first $2n$ steps that are taken in the $i$-th coordinate. For $n_1+\cdots+n_d=n$ with $n_i\in\mathbb{Z}_{+}$ (for all $i$),
$$
\P (S_i = 2 n_i, \ 1\le i\le d)
=
d^{-2n} {2n \choose 2n_1, \ldots, 2n_d}.
$$

\noindent By \cite[Lemma 1.4]{PG2015}, there exist constants $c_3>0$ and $c_4>0$, depending only on $d$, such that
$$
\sum_{\substack{n_1+\cdots + n_d=n \\ \exists n_i \not\in [\frac{n}{d},\, \frac{3n}{d}]}} \P (S_i = 2 n_i, \ 1\le i\le d)
\le
c_3\exp(-c_4n).
$$

\noindent Hence
\begin{align*}
 &d^{-2n} \sum_{\substack{n_1 + \cdots + n_d = n \\ n_i \ge 1, \; 1\le i\le d}} {2n \choose 2n_1, \ldots, 2n_d}
    \prod_{j=1}^d n_j^{-3/2}
    \\
 &\le c_3\exp(-c_4n)
    +
    \sum_{\substack{n_1+\cdots + n_d=n \\ n_i \in [\frac{n}{d},\, \frac{3n}{d}], \; 1\le i\le d}} \P (S_i = 2 n_i, \ 1\le i\le d)
    \prod_{j=1}^d n_j^{-3/2} \, .
\end{align*}

\noindent Consider the sum on the right-hand side. Since $n_i \in [\frac{n}{d},\, \frac{3n}{d}]$ for all $1\le i\le d$, we argue that $\prod_{j=1}^d n_j^{-3/2} \le (\frac{d}{n})^{3d/2}$, so the sum is bounded by $(\frac{d}{n})^{3d/2} \sum_{\substack{n_1+\cdots + n_d=n \\ n_i \in [\frac{n}{d},\, \frac{3n}{d}], \; 1\le i\le d}} \P (S_i = 2 n_i, \ 1\le i\le d) \le (\frac{d}{n})^{3d/2}$. Consequently,
$$
d^{-2n} \sum_{\substack{n_1 + \cdots + n_d = n \\ n_i \ge 1, \; 1\le i\le d}} {2n \choose 2n_1, \ldots, 2n_d}
\prod_{j=1}^d n_j^{-3/2}
\le
c_3\exp(-c_4n)
+
(\frac{d}{n})^{3d/2}
\le
c_5\, n^{-3d/2} \, ,
$$

\noindent for some constant $c_5>0$ depending on $d$. Going back to \eqref{e:Pgammalambda2}, we obtain
$$
\sum_{\gamma \in \sP_{2n}\setminus\sP_{2n}^0}
\P(\gamma, \, \lambda)
\le
c_2 c_5 \Big( \frac{2\sqrt{\lambda}}{1 + \lambda} \Big)^{\! 2n} \, n^{-3d/2}.
$$

\noindent In view of \eqref{e:P2n0}, and since $\eta <1$ and $p^{(2n)}_{\lambda}(o, \, o) = \sum_{\gamma \in \sP_{2n}} \P(\gamma, \, \lambda)$, this yields the desired upper bound for $p^{(2n)}_{\lambda}(o, \, o)$.
\end{proof}


Let $\sX := \{ (x_1, \ldots, x_d) \in \Z^d: \, x_i =0 \hbox{ \rm for some } i\}$.

\begin{lem}
 \label{P:axialplane}

 Almost surely, $\RW_{\lambda}$ with $\lambda\in (0, \, 1)$ visits $\sX$ only finitely many times.

\end{lem}

\begin{proof} In dimension $d=2$, the lemma is a consequence of \cite[Proposition 2.1]{KI-MV1998}, whose proof relies on properties of Riemann surfaces, and does not seem to be easily extended to higher dimensions.

Let $\lambda\in (0,\, 1)$. Let $(X_n)_{n=0}^{\infty} := ( (X_n^1, \ldots, X_n^d ))_{n=0}^{\infty}$ be $\RW_{\lambda}$ on $\Z^d$. Write $Y_n := (|X_n^1|, \ldots, |X_n^d|)$ for $n\in\mathbb{Z}_{+}$. Then $(Y_n)_{n=0}^{\infty}$ is a Markov chain on the first orthant $\Z^d_+$. Define
$$
\sigma_1
:=
\inf \{ n > 0: \ Y_n \in \Z^d_+ \setminus \sX \},
\qquad
\tau_1
:=
\inf \{ n > \sigma_1 : \ Y_n \in \sX \},
$$

\noindent and recursively for $i\ge 2$,
$$
\sigma_i
:=
\inf \{n > \tau_{i-1} :\ Y_n \in \Z^d_+ \setminus \sX \},
\qquad
\tau_i
:=
\inf \{n> \sigma_i:\ Y_n \in \sX \} ,
$$

\noindent with the convention that $\inf \emptyset := \infty$. Let
$(\sF_n)_{n=0}^{\infty}$ be the filtration generated by $(Y_n)_{n=0}^{\infty}$, i.e., $\sF_n := \sigma(Y_1, \ldots, Y_n)$. We claim that
\begin{enumerate}[{\bf (i)}]
  \item For any $i>1$, conditioned on $\{\tau_{i-1}<\infty\}$ and $\sF_{\tau_{i-1}}$, $\sigma_i < \infty$ a.s.
  \item There exists a constant $0<q<1$ such that for any $i\ge 1$, $\P (\tau_i<\infty \, |\, \sigma_i<\infty, \; \sF_{\sigma_i} ) \le q$.
\end{enumerate}

Indeed, conditionally on $\{\tau_{i-1} < \infty\}$ and $\sF_{\tau_{i-1}}$, $(Y_{\tau_{i-1} + n})_{n=0}^{\infty}$ is a Markov chain starting at $Y_{\tau_{i-1}}$ with the same transition probability as that of $(Y_n)_{n=0}^{\infty}$. At each step, the transition probability from a state in $\sX \setminus \{o\}$ to another state in $\sX$ is $\frac{d-k + (d-k) \lambda}{d+k + (d-k) \lambda}$ for some $1 \le k \le d-1$, which is at most $\frac{(d-1) (1+\lambda)}{(d-1) (1+\lambda) + 2} < 1$. Since the number of visits to $o$ in the first $2n$ steps is at most $n$, we have
\begin{align*}
 &\P (\sigma_i-\tau_{i-1}>2n\, | \, \tau_{i-1}<\infty, \; \sF_{\tau_{i-1}})
    \\
 &=\P (Y_{\tau_{i-1}+k}\in \sX \ \text{for}\ 1 \le k \le 2n \, | \,\tau_{i-1}<\infty, \; \sF_{\tau_{i - 1}} )
   \le \left( \frac{(d-1) (1+\lambda)}{(d-1) (1+\lambda) + 2}\right)^n.
\end{align*}

\noindent We get (i) by sending $n$ to $\infty$.

Let $(Z_n)_{n=0}^{\infty}$ be a drifted random walk on $\Z^d$, starting inside the first open orthant, with the step distribution $\mu$ given by $\mu(\bfe_1) = \cdots = \mu(\bfe_d) = \frac{1}{d(1+\lambda)}$ and $\mu(-\bfe_1) =  \cdots =\mu(-\bfe_d) = \frac{\lambda}{d(1+\lambda)}$ (where $\{\bfe_1, \ldots, \bfe_d\}$ is the standard basis in $\Z^d$). Let $\tau := \inf \{n\geq 0:\ Z_n\in\sX \}$. Since the walk has a constant drift whose components are all strictly positive, $\P(\tau<\infty) \le q <1$ where $q$ depends on $d$ and $\lambda$.

Conditioned on $\sigma_i<\infty$ and $\sF_{\sigma_i}$, $(Y_{\sigma_i+n}, \, 0\le n<\tau_i-\sigma_i)$ has the same distribution as $(Z_n, 0\le n<\tau)$. Now (ii) follows readily.

By (i) and (ii), for $i\ge 2$, $\P(\tau_i < \infty) \le q \, \P(\tau_{i-1} < \infty)$, hence $\P(\tau_i < \infty) \le q^i$. The Borel--Cantelli lemma implies that a.s.\ there are only finitely many $i$'s such that $\tau_i < \infty$. Let $m$ be the largest one. The total number of visits to $\sX$ of $(Y_n)_{n=0}^{\infty}$ is $\sigma_1 + (\sigma_2 - \tau_1) + \cdots + (\sigma_m - \tau_{m-1})$, which is a.s.\ finite.
\end{proof}


\begin{thm}\label{T:speed}
Let $\lambda \in (0, \, 1)$ and let $(X_n)$ be $\RW_\lambda$ on $\Z^d$. Then
$$
\lim_{n\to\infty} \frac{1}{n} \left(\left|X_n^1\right|,
\ldots, \left|X_n^d\right|\right) = \frac{1-\lambda}{1+\lambda}
     \left(\frac1d, \, \ldots, \, \frac1d \right) 
\qquad
\hbox{\rm a.s.}
$$

\noindent In particular, the speed $\sS(\lambda)=\frac{1-\lambda}{1+\lambda}$ of $\RW_{\lambda}$ is positive and strictly decreasing in $\lambda \in (0,\, 1)$.

\end{thm}

\begin{proof}
For simplicity, we only prove the theorem for $d=2$. Define functions $f_1$ and $f_2$ on $\Z^2$ by
$$
f_1(x)
:=
    \begin{cases}
      0, & x_1 = 0, \\
      \frac{1-\lambda}{3+\lambda}, & x_1 \neq 0, \ x_2 = 0, \\
      \frac{1-\lambda}{2(1+\lambda)}, & \text{otherwise},
    \end{cases}
\qquad
f_2(x)
:=
    \begin{cases}
      \frac{1-\lambda}{3+\lambda}, & x_1 = 0, \ x_2 \neq 0, \\
      0, & x_2 = 0, \\
      \frac{1-\lambda}{2(1+\lambda)}, & \text{otherwise}.
    \end{cases}
$$

\noindent It is easily seen that $( |X_n^1| - |X_{n-1}^1| - f_1(X_{n-1}), \, |X_n^2| - |X_{n-1}^2| - f_2(X_{n-1}) )_{n=1}^{\infty}$ is a martingale-difference sequence. By the strong law of large numbers (cf.\ \cite[Theorem 13.1]{LR-PY2016}),
$$
\lim_{n \to \infty} \frac1n \left( |X_n^1| - \sum_{k=0}^{n-1} f_1(X_k) \right)
=
\lim_{n \to \infty} \frac1n \left( |X_n^2| - \sum_{k=0}^{n-1} f_2(X_k) \right)
=
0
\qquad
\hbox{\rm a.s.}
$$

\noindent Since $|X_n| = |X_n^1| + |X_n^2|$, the theorem follows from Lemma \ref{P:axialplane} and the definitions of $f_1$ and $f_2$.
\end{proof}

\section{Intersections of two independent random walks}
\label{s:intersection}

In this section, we consider the number of intersections of two
independent drifted or biased random walks on $\Z^d$. As we mentioned in the introduction, these results are crucial in the forthcoming computation in Section~\ref{s:usf} of the number of trees in the uniform spanning forests of $\Z^d$.

\subsection{Intersections of drifted random walks: Proof of Theorem~\ref{T:intersection}}

Let $(Z_n)_{n=0}^{\infty}$ and $(W_n)_{n=0}^{\infty}$ be two independent drifted random walks on $\Z^d$ with the same step distribution $\mu$ given by \eqref{e:mu}, starting at $z_0$ and $w_0$ respectively.

Without loss of generality, let us assume $z_0=w_0=0$. The expectation of the intersection number for $(Z_m)_{m=0}^{\infty}$ and $(W_n)_{n=0}^{\infty}$ is
\begin{equation}
    \label{e:expectation}
    \sum_{m=0}^{\infty} \sum_{n = 0}^{\infty}
    \P ( Z_m = W_n )
    =
    \sum_{m=0}^{\infty} \sum_{n=0}^{\infty} \sum_{x \in \Z^d}
    p^{(m)}(o,\, x) \, p^{(n)}(o,\, x),
\end{equation}

\noindent where $p^{(n)}(x,\, y)$ is the $n$-step transition probability for $(Z_m)_{m=0}^{\infty}$ from $x$ to $y$. By
\cite[Theorem 10.24]{LR-PY2016}, to prove Theorem \ref{T:intersection}, it suffices to prove that the sum on the right-hand side of \eqref{e:expectation} is finite if $d \ge 4$, and is infinite if $d \le 3$.

Let $\mathbf{m}$ and $\Sigma = (\Sigma_{ij})$ be respectively the mean and the covariance matrix of $\mu$.  Then $\mathbf{m} = \frac{1-\lambda}{d(1+\lambda)}(1, \ldots, 1)$ and $\Sigma_{ij} = \frac{1}{d} \delta_{ij} - \frac{(1-\lambda)^2}{d^2 (1+\lambda)^2}$ for $1\le i, \, j \le n$. By the local limit theorem (\cite[Theorem~2]{SC1966}),\begin{equation}
	\label{e:llt}
	p^{(n)}(o,\, x)
	=
	\frac{1}{(2 \pi n)^{d/2} (\det \Sigma)^{1/2}} \exp \left( - \frac{ (x - n \mathbf{m}) \cdot \Sigma^{-1} (x - n \mathbf{m})}{2n} \right) + o(n^{-d/2}),
\end{equation}

\noindent where $n^{d/2} o(n^{-d/2}) \to 0$ as $n \to \infty$ uniformly in $x \in \Z^d$. Since the largest eigenvalue of $\Sigma$ is $\frac1d$, we have
$$
(x - n \mathbf{m}) \cdot \Sigma^{-1} (x - n \mathbf{m})
\ge
d | x - n \mathbf{m}|
$$

\noindent for $x\in \Z^d$. The local limit theorem \eqref{e:llt} immediately implies the following result.

\begin{lem}
  \label{L:hkdriftupper}

{\bf (i)} There exists a constant $c>0$ such that
    \begin{equation}
      \label{e:hkup}
      \sup_{x\in \Z^d} p^{(n)}(0, \, x) \le
      c n^{-d/2},
      \qquad
      \forall n\ge 1.
    \end{equation}

  \label{L:hklower}

{\bf (ii)} For $\sigma>0$, define
 $$
 R_{n,\sigma}
 :=
 \{ x\in \Z^d :\ | x_i - \frac{1-\lambda}{d(1+\lambda)} n | \le \sigma n^{1/2},\ 1\le i\le d \}.
 $$

 \noindent Then there exists a constant $c>0$, depending on $\sigma$, $\lambda$ and $d$ such that for any $n\in\mathbb{N}$ with $R_{n,\sigma} \not= \varnothing$,
 \begin{equation}
     \label{e:hklower}
     p^{(n)}(0,\, x)
     \ge
     c n^{-d/2}
     \qquad
     \text{for $x \in R_{n,\sigma}$ with $n + |x|$ being even.}
 \end{equation}

\end{lem}

We need another preliminary result.

\begin{lem}
 \label{L:concatenation}

 Let $\varepsilon > 0$. For any $n\in\mathbb{N}$, define
 $$
 Q_n(\varepsilon)
 :=
 \{ x = (x_1,\cdots, x_d) \in \Z^d :\ | x_i - \frac{1-\lambda}{d(1+\lambda)} n | < n^{(1+\varepsilon)/2},\ 1\le i\le d \}.
 $$

 \noindent Then there exists a constant $c>0$, depending on $\lambda\in (0, \, 1)$ and $d$, such that
 \begin{equation}
     \sum_{x \in \Z^d \setminus Q_n(\varepsilon)} p^{(n)}(0,\, x)
     \le
     2 d \exp (-c n^{\varepsilon} ),
     \qquad
     \forall n\in\mathbb{N}.
     \label{e:concentration}
 \end{equation}

\end{lem}

\begin{proof}
By the Azuma--Hoeffding inequality, there exists a constant
$c_0>0$, depending only on $\lambda$ and $d$, such that
$$
\P ( \max_{1\le i\le d} |Z_n^i - \frac{1-\lambda}{d(1+\lambda)} n | \ge t )
\le
2 d \exp ( - \frac{c_0 t^2}{n} ),
\qquad
t > 0 ,
$$

\noindent where $Z_n^i$ is the $i$-th coordinate component of $Z_n\in\Z^d$. The lemma follows by taking $t = n^{(1+\varepsilon)/2}$.
\end{proof}

Now we are ready to prove Theorem \ref{T:intersection}.

\begin{proof}[Proof of Theorem \ref{T:intersection}]
\textbf{Case 1: $d \ge 4$.}

Fix a small enough $\varepsilon\in (0,\, 1)$. Let $n_{\varepsilon} := \max \{ n - \frac{2d (1+\lambda)}{1-\lambda} n^{(1+\varepsilon)/2}, \, 0\}$. Note that if $1\le m < n_{\varepsilon}$, then
$$
\frac{1 - \lambda}{d(1+\lambda)} n - n^{(1+\varepsilon)/2}
>
\frac{1-\lambda}{d(1+\lambda)}m + m^{(1+\varepsilon)/2}.
$$

\noindent This implies $Q_m(\varepsilon)\cap Q_n(\varepsilon) =\varnothing$. In particular,
$$
\sum_{n\in\mathbb{N}} \sum_{1\le m < n_{\varepsilon}}
\sum_{x \in \Z^d} p^{(m)}(0,\, x) p^{(n)}(0,\, x)
\le
\sum_{n\in\mathbb{N}} \sum_{1\le m < n_{\varepsilon}}
\left( \sum_{x \in \Z^d \setminus Q_n(\varepsilon)} + \sum_{x \in \Z^d \setminus Q_m(\varepsilon)} \right)
p^{(m)}(0,\, x) p^{(n)}(0,\, x) .
$$

\noindent By Lemmas \ref{L:hkdriftupper} and \ref{L:concatenation}, $\sum_{x \in \Z^d \setminus Q_n(\varepsilon)} p^{(m)}(0,\, x) p^{(n)}(0,\, x)$ and $\sum_{x \in \Z^d \setminus Q_m(\varepsilon)} p^{(m)}(0,\, x) p^{(n)}(0,\, x)$ are bounded by $2d \exp (-c_2 n^{\varepsilon} ) c_1m^{-d/2}$ and $2d \exp (-c_2 m^{\varepsilon} ) c_1 n^{-d/2}$, respectively. Hence
\begin{align*}
 &\sum_{n\in\mathbb{N}} \sum_{1\le m < n_{\varepsilon}}
\sum_{x \in \Z^d} p^{(m)}(0,\, x) p^{(n)}(0,\, x)
    \\
 &\le \sum_{n\in\mathbb{N}} \sum_{1\le m < n_{\varepsilon}}
    \Big( 2d \exp (-c_2 n^{\varepsilon} ) c_1m^{-d/2} + 2d \exp (-c_2 m^{\varepsilon} ) c_1 n^{-d/2} \Big)
    < \infty.
\end{align*}

\noindent On the other hand, by Lemma \ref{L:hkdriftupper},
\begin{align*}
    \sum_{n\in\mathbb{N}} \sum_{n_{\varepsilon}\le m\le n}
    \sum_{x \in \Z^d} p^{(m)}(0,\, x) p^{(n)}(0,\, x)
 &\le \sum_{n\in\mathbb{N}} \sum_{n_{\varepsilon}\le m\le n}
    \sum_{x \in \Z^d} p^{(m)}(0,\, x) c_1 n^{-d/2}
    \\
 &= \sum_{n\in\mathbb{N}} \sum_{n_{\varepsilon}\le m\le n}
c_1 n^{-d/2}
    \le
    \sum_{n\in\mathbb{N}} c_3 n^{- (d-1-\varepsilon)/2}
    <\infty.
\end{align*}

\noindent Moreover, by transience of $(Z_n)_{n=0}^{\infty}$,
$$
\sum_{n\in\mathbb{N}} \sum_{x \in \Z^d} p^{(1)}(0,\, x) p^{(n)}(0,\, x)
\le
\sum_{n\in\mathbb{N}} \sum_{x \in \Z^d} p^{(1)}(0,\, x) \, c_1 n^{-d/2}
=
\sum_{n\in\mathbb{N}} c_1 n^{-d/2}
<\infty.
$$

\noindent Assembling these pieces yields $\sum_{n\in\mathbb{N}} \sum_{m\in\mathbb{N}} \sum_{x \in \Z^d} p^{(m)}(0,\, x) p^{(n)}(0,\, x) <\infty$. A fortiori, we obtain $\sum_{m=0}^{\infty} \sum_{n = 0}^{\infty} {\bf 1}_{\{ Z_m = W_n\} } <\infty$ a.s., as desired.

\textbf{Case 2: $d \le 3.$}

By \eqref{e:hklower} in Lemma \ref{L:hkdriftupper}, there exist constants $c_4>0$ and $c_5>0$ such that
\begin{align*}
    \sum_{n=0}^{\infty}\sum_{m=0}^{\infty}
    \sum_{x \in \Z^d}
    p^{(n)}(0,\, x) p^{(m)}(0,\, x)
 &\ge \sum_{n=2}^{\infty} \sum_{n - n^{1/2}\le m\le n}
    \sum_{x \in R_{n,1}}
    p^{(n)}(0,\, x) p^{(m)}(0,\, x)
    \\
 &\ge \sum_{n=2}^{\infty} \sum_{n - n^{1/2}\le m\le n}
    \sum_{x \in R_{n,1}}
    c_4\, n^{-d/2} m^{-d/2}
    \\
 &\ge c_5 \sum_{n=2}^{\infty} \sum_{n - n^{1/2}\le m\le n}
    m^{-d/2}
    \ge
    c_5\sum_{n=2}^{\infty} n^{-\frac{d-1}{2}},
\end{align*}

\noindent which is infinity. By \cite[Theorem 10.24]{LR-PY2016}, $\sum_{m=0}^{\infty} \sum_{n = 0}^{\infty} {\bf 1}_{\{ Z_m = W_n\} } =\infty$ a.s.
\end{proof}

\subsection{Intersections of biased random walks}

For $x = (x_1, \ldots, x_d) \in \Z^d$, define $\phi(x) := (|x_1|, \ldots, |x_d|) \in \Z_{+}^d$. We start by studying the number of intersections of the reflecting random walks $(\phi(X_n))_{n=0}^{\infty}$ and $(\phi(Y_n))_{n=0}^{\infty}$,
where $(X_n)_{n=0}^{\infty}$ and $(Y_n)_{n=0}^{\infty}$ are independent $\RW_\lambda$'s on $\Z^d$. By Theorem~\ref{T:intersection}, with positive probability, the number of intersections is infinite if $d \le 3$, and is finite if $d \ge 4$. The Liouville property below will ensure that the probability
is indeed one.

\begin{lem}
  \label{L:poisson}
  Let $(X_n)_{n=0}^{\infty}$ be $\RW_{\lambda}$ on $\Z^d$ with $\lambda\in (0,\, 1)$. The Poisson boundary for $(\phi(X_n))_{n=0}^{\infty}$ is trivial, i.e., all bounded harmonic functions are constants.

\end{lem}

\begin{proof}
When $d = 2$, the lemma is a special case of the main result in \cite{KI-1999}. Our proof is essentially a reproduction of the argument of \cite{KI-1999}, formulated for all $d$.

Following \cite{BD1955}, a subset $C \subset \Z^d_+$ is said to be \emph{almost closed} with respect to $(\phi(X_n))_{n=0}^{\infty}$ if
$$
\P(\phi(X_n) \in C \text{ for all sufficiently large } n) = 1.
$$

\noindent A set $C$ is called \emph{atomic} if $C$ does not contain two disjoint almost closed subsets. By \cite{BD1955}, there exists a collection $\{C_1, \, C_2, \cdots \}$ of disjoint almost closed sets such that
\begin{eqnarray*}
&&{\bf (i)}\ \mbox{every $C_i$ except at most one is atomic},\\
&&{\bf (ii)}\ \mbox{the non-atomic $C_i$, if present, contains no atomic subsets,}\\
&&{\bf (iii)}\ \sum_i \P ( \lim_{n\to \infty} \{ \phi(X_n) \in C_i \} ) =1.
\end{eqnarray*}

\noindent Furthermore, $(\phi(X_n))_{n=0}^{\infty}$ has the Liouville property if and only if it is simple and atomic in the sense that the decomposition consists of a single atomic set $C_1$.

Let
$$
\sX  := \{ x = (x_1, \ldots, x_d) \in \Z^d_+ : \
  x_i = 0 \text{ for some } 1 \le i \le d \},
$$

\noindent be the boundary of $\Z^d_+$. Starting at $x \in \Z^d_+ \setminus \sX$, $(\phi(X_n))_{n=0}^{\infty}$ has the same
distribution as the drifted random walk $(Z_n)_{n=0}^{\infty}$ driven by $\mu$ specified in Theorem \ref{T:intersection}, before hitting the boundary $\sX$.

Define
$$
\tau^X
:=
\inf \{n :\ \phi(X_n) \in \sX \}
\ \mbox{and}\
\tau^Z = \inf \{n:\ Z_n \in \sX\}.
$$

\noindent Let $J \subset \Z^d_+ \setminus \sX$ be an almost closed set with respect to $(\phi(X_n))_{n=0}^{\infty}$, i.e.,
$$
\P(\phi(X_n) \in J \text{ for all sufficiently large } n ) = 1.
$$

\noindent Since $( \phi(X_n), \, n < \tau^X)$ is distributed as $(Z_n, \, n < \tau^Z$, and $\P(\tau^X = \infty)> 0$, we have
\begin{align}
 &\P (Z_n \in J \text{ for all sufficiently large } n )
    \nonumber
    \\
 &\ge \P (\phi(X_n)\in J \text{ for all sufficiently large }n,\ \tau^X = \infty ) > 0.
    \label{e:positive}
\end{align}

\noindent By \cite[Theorem 3]{BD1955}, $(Z_n)_{n=0}^{\infty}$ has the Liouville property, thus
$$
\P(Z_n \in J \text{ for all sufficiently large } n) \in\{0, \, 1\}.
$$

\noindent Combining this with \eqref{e:positive}, we see that $J$ is
also almost closed with respect to $(Z_n)_{n=0}^{\infty}$.

Thus, the decomposition for $(\phi(X_n))_{n=0}^{\infty}$ is automatically the unique decomposition for $(Z_n)_{n=0}^{\infty}$. Since $(Z_n)_{n=0}^{\infty}$ is simple and atomic, so is $(\phi(X_n))_{n=0}^{\infty}$, which is equivalent to the aforementioned
Liouville property.
\end{proof}

\begin{lem}
  \label{L:intersect2}
 Let $(X_n)_{n=0}^{\infty}$ and $(Y_n)_{n=0}^{\infty}$ be independent $\RW_\lambda$'s on $\Z^d$ with $\lambda\in (0, \, 1)$. Then almost
 surely the number of intersections of $(\phi(X_n))_{n=0}^{\infty}$
 and $(\phi(Y_n))_{n=0}^{\infty}$ is infinite if $d\le 3$ and is finite if $d\ge 4$.

\end{lem}

\begin{proof}
By Lemma \ref{L:poisson} and the proof of \cite[Theorem
1.1]{BI-CN-GA2012}, the probability that $(\phi(X_n))_{n=0}^{\infty}$ and $(\phi(Y_n))_{n=0}^{\infty}$ intersect infinitely often is either $0$ or $1$.

Before hitting any axial hyperplanes, $(\phi(X_n))_{n=0}^{\infty}$ and $(\phi(Y_n))_{n=0}^{\infty}$ has the same joint distribution as that of $(Z_n)_{n=0}^{\infty}$, $(W_n)_{n=0}^{\infty}$, where
$(Z_n)_{n=0}^{\infty}$ and $(W_n)_{n=0}^{\infty}$ are independent drifted random walks on $\Z^d$ with step distribution $\mu$ described  in Theorem \ref{T:intersection}, and $Z_0=\phi(X_0)$, $W_0=\phi(Y_0)$.
Let $T$ be the first time either $(Z_n)_{n=0}^{\infty}$ or $(W_n)_{n=0}^{\infty}$ hits an hyperplane. By Theorem \ref{T:intersection}, on the event $\{T =\infty\}$, which has positive
probability, the number of intersections between $(Z_n)_{n=0}^{\infty}$ and $(W_n)_{n=0}^{\infty}$ is infinite if $d\le 3$, and is finite if $d\ge 4$. In view of the aforementioned $0$--$1$ law above prove this lemma.
\end{proof}

\section{Uniform spanning forests associated with $\RW_\lambda$}
\label{s:usf}

Let $G=(V(G), E(G))$ be a locally finite, connected infinite graph, rooted at $o$. To each edge $e=(x,\, y) \in E(G)$, we assign a weight or conductance $c(e) = c(x,y) = c(y,\, x)$. The weighted graph $(G,\, c)$ is called an electrical network. Consider a Markov chain on $G$ with transition probability $p(x,\, y) = \frac{c(x,\, y)}{\sum_{z \sim x} c(x,\, z)}$, where $z \sim x$ means that $z$ and $x$ are adjacent vertices in $G$. The chain is referred to as a random walk on $G$ with conductance $c$.
Biased random walk $\RW_\lambda$ on $G$ is a random walk on $G$ with conductance defined by $c(e) = c_\lambda(e) := \lambda^{-|e|}$.

For any finite network $(G,\, c)$, we consider associated spanning trees, i.e., subgraphs that are trees and that include every vertex. We define the uniform spanning tree measure $\UST_G$ to be the probability measure on spanning trees of $G$ such that the measure of each tree is proportional to the product of conductances of the edges in the tree.

An exhaustion of an infinite graph $G$ is a sequence $\{ V_n \}_{n \geq 1}$ of finite, connected subsets of $V(G)$ such that $V_n \subset V_{n+1}$ for all $n\ge 1$ and $\cup_n V_n = V(G)$. Given such an exhaustion, we define the network $G_n$ to be the subgraph of $G$ induced by $V_n$ together with the conductances inherited from $G$. The free uniform spanning forest measure $\FSF$ is defined to be the
weak limit of the sequence $\{ \UST_{G_n} \}_{n \ge 1}$ in the sense that
$$
\FSF ( S \subset \F )
=
\lim_{n\to \infty} \UST_{G_n} (S \subset T) ,
$$

\noindent for each finite set $S \subset E(G)$. For each $n$, we can also construct a network $G_n^*$ from $G$ by gluing (= wiring) every vertex of $G \setminus G_n$ into a single vertex, denoted by $\partial_n$, and deleting all the self-loops that are created. The set of edges of $G_n^*$ is identified with the set of edges of $G$ having at least one endpoint in $V_n$. The wired uniform spanning forest measure $\WSF_G$ is defined to be the weak limit of the sequence $\{\USF_{G_n^*}\}_{n\ge 1}$ so that
$$
\WSF_G (S \subset \F)
=
\lim_{n\to\infty} \UST_{G_n^*} (S \subset T),
$$

\noindent for each finite set $S \subset E(G)$. [For the existence of both $\FSF$ and $\WSF$, see \cite[Chapter 10]{LR-PY2016}.] Both measures $\FSF$ and $\WSF$ are easily seen to be concentrated on the set of uniform spanning forests of $G$ with the property that every
connected component is infinite. It is also easy to see that $\FSF$ stochastically dominates $\WSF$ for any infinite network $G$.

The number of trees in the wired uniform spanning forest is a.s.\ a constant; see \eqref{e:treenum} below. In \cite{BI-LR-PY2001}, it is asked whether $\FSF$ and $\WSF$ are mutually singular (also formulated in \cite[Question 10.59]{LR-PY2016}) and whether the number of trees in the free uniform spanning forest is a.s.\ constant (also formulated in \cite[Question 10.28]{LR-PY2016}) if $\FSF \not= \WSF$.\footnote{For a group acting on a network so that every vertex has an infinite orbit, it is known (\cite[Corollary 10.19]{LR-PY2016}) that the action is mixing and ergodic for both $\FSF$ and $\WSF$ (so the number of trees in the uniform spanning forest is a.s.\ a constant), and if $\FSF$ and $\WSF$ are distinct, they are mutually singular. It is unknown whether this remains true without the assumption that each vertex has an infinite orbit.}
To answer these questions, the first step is to know whether $\FSF$ and $\WSF$ are identical. 
When the electric network is not transitive and $\FSF \not= \WSF$, it seems interesting to study whether $\FSF$ and $\WSF$ are singular. A simple situation is when $G$ is a tree, in which case the free uniform spanning forest has one tree (which is the singleton $\{ G\}$), whereas the number of trees in the wired uniform spanning forest can be higher if the constant $K$ defined in \eqref{e:treenum} below is at least 2.

Let $\lambda > 0$. Let $c_\lambda(\cdot)$ be the conductances associated with $\RW_\lambda$ on graph $G$, and $r_\lambda(\cdot) :=\frac{1}{c_\lambda(\cdot)}$ being the corresponding resistance. Write $\FSF_\lambda$ and $\WSF_\lambda$ for the free and wired uniform spanning forest measures. [When they are identical, we use the notation $\USF_\lambda$ instead.]

We give a criterion to determine whether $\FSF_\lambda = \WSF_\lambda$, compute the number of trees in $\USF_\lambda$ on $\Z^d$, and consider the singularity problem when $\FSF_\lambda \not= \WSF_\lambda$.

\subsection{$\USF_\lambda$ on $\Z^d$}

On any graph $G$, if $\lambda > \lambda_c(G)$, then $\RW_\lambda$ is recurrent, so $\FSF_\lambda = \WSF_\lambda$. The following theorem deals with the case $0< \lambda < \lambda_c(G)$. Recall (\cite[Section 6.5]{LR-PY2016}) that a graph is said to have one end if the deletion of any finite set of vertices leaves exactly one infinite component.

\begin{thm}
\label{thm6.1}

Let $G$ be a graph with one end such that
\begin{equation}
    \label{(6.1)}
    \lim_{n\to\infty} \Big( \sum_{x\in\partial B_G(n)} (d_x^{+}+d_x^0)\Big)^{1/n}
    =1.
\end{equation}

\noindent Then for $0<\lambda<\lambda_c(G)=1$ 
we have $\FSF_\lambda = \WSF_\lambda$. In particular, for any $d\ge 2$ and any Cayley graph of additive group $\mathbb{Z}^d$, $\FSF_\lambda = \WSF_\lambda$ for $\lambda\in (0,\, 1)$.

\end{thm}

The proof of Theorem \ref{thm6.1} shows that for any graph $G$ with one end and such that
$$
\mathrm{gr}_{*}(G)
:=
\limsup_{n\to\infty} \Big( \sum_{x\in\partial B_G(n)} (d_x^{+}+d_x^0) \Big)^{1/n}
\in
[1, \, \infty),
$$

\noindent 
we have $\FSF_\lambda = \WSF_\lambda$ for any $0<\lambda<\frac{1}{\mathrm{gr}_{*}(G)^2}$. 





\begin{proof}[{\bf Proof of Theorem \ref{thm6.1}}]
For any function $f: \, V \to \R$, let ${\rm d}f$ be the antisymmetric function on oriented edges defined by
$$
{\rm d} f(e)
:=
f(e^-) - f(e^+),
$$

\noindent where $e^-$ and $e^+$ are respectively the tail and head of
$e$. Define the space of Dirichlet functions as
$$
{\mathbf D}_\lambda
:=
\Big\{ f: \ ({\rm d}f,{\rm d}f)_{c_\lambda} := \sum_{e\in \mathbf{E}} |{\rm d}f(e) |^2c_\lambda(e)<\infty \Big\} ,
$$

\noindent where $\mathbf{E}$ is the set of all oriented edges of $G$. 
By
\cite[Theorem 7.3]{BI-LR-PY2001}, 
\begin{align*}
    \FSF_\lambda = \WSF_\lambda\
 &\; \Longleftrightarrow \;
    \mbox{all harmonic functions in $\mathbf{D}_\lambda$ are constant}.
\end{align*}

Clearly $\lambda_c(G)=1$. Let $\lambda\in (0,\, 1)$. Let $f$ be a harmonic function in $\mathbf{D}_\lambda$. We need to prove that $f$ is a constant.

By the maximum principle, for every $n\ge 1$, there are $v_1(n)$, $v_2(n)\in \partial B_G(n)$ such that $f$ takes its maximum
at $v_1(n)$ and minimum at $v_2(n)$ over all vertices in $B_G(n)$. By the assumption,
$$
({\rm d}f, \, {\rm d}f)_{c_\lambda}
=
\sum_{e\in\mathbf{E}} |{\rm d}f(e)|^2 \, \lambda^{-|e|}
<
\infty,
$$

\noindent where $|e|$ is the distance for $e$ from $o$. Hence for some constant $C>0$, $\sup_{e\in\mathbf{E}} |{\rm d}f(e)|^2 \lambda^{-|e|} \le C$, i.e.,
$$
|{\rm d}f(e)|
\le
C^{1/2}\lambda^{|e|/2},
\qquad
\forall e\in\mathbf{E}.
$$

\noindent Combined with (\ref{(6.1)}), we see that
\begin{equation}
    \label{(6.2)}
    \sum_{e\in\mathbf{E}} |{\rm d}f(e)|
    \le
    C^{1/2} \sum_{n=0}^\infty \lambda^{n/2}\sum_{e\in\mathbf{E}, \; |e|=n} 1
    <
    \infty.
\end{equation}

Let $n\ge 1$. Since $G$ has one end, $G\setminus B_G(n)$ is a connected graph, so there is a finite path $u_0^nu_1^n\cdots u_{k_n}^n$ in $G\setminus B_G(n)$ such that $u_0^n=v_1(n+1)$, $u_{k_n}^n=v_2(n+1)$. As such,
$$
0
\le
f(v_1(n+1))-f(v_2(n+1))
=
\sum_{j=1}^{k_n} [ f(u_{j-1}^n)-f(u_j^n) ]
\le
\sum_{e\in\mathbf{E}, \, |e| \ge n+1} |{\rm d}f(e)|.
$$

\noindent By (\ref{(6.2)}),
$$
\lim_{n\to\infty} \{f(v_1(n+1))-f(v_2(n+1))\}=0,
$$

\noindent which implies that $f$ is constant.
\end{proof}

Let us consider the uniform spanning forests associated with
$\RW_{\lambda}$ on $\Z^d$. Theorem \ref{thm6.1} says that
$\FSF_{\lambda} = \WSF_{\lambda}$ on $\Z^d$ ($d\ge 2$) for $\lambda \in (0, \, 1)$. For $\lambda = 1$, the two measures are also known to be identical (
\cite{PR1991}). In these cases, we denote both of them by $\USF_{\lambda}$.

When $\lambda=1$, the uniform spanning forest on $\Z^d$ has one tree a.s.\ for $d \le 4$ and has infinitely many trees a.s.\ for $d \ge 5$; see \cite{PR1991} or \cite[Theorem 10.30]{LR-PY2016}. When $0<\lambda<1$, Theorem \ref{thm6.2} below reveals the existence of a novel phase transition, with the critical dimension reduced to $3$. 

\begin{thm}\label{thm6.2}

Let $0 < \lambda < 1$.

{\bf (i)} Almost surely, the number of trees in the uniform spanning forest associated with $\RW_{\lambda}$ on $\Z^d$ is $2^d$ if $d = 2$ or $3$, 
and is infinite if $d \ge 4$. Moreover, when $d\ge 2$, $\USF_\lambda$-a.s.\ every tree has one end.

{\bf (ii)} On $\mathbb{Z}^1$, $\FSF_\lambda \not= \WSF_\lambda$: the free uniform spanning forest is the singleton of the tree $\mathbb{Z}^1$, whereas the wired uniform spanning forest has two trees and satisfies
\begin{equation}
    \WSF_{\lambda} [\mathfrak{F} = \{T_{i-1}^{-},\, T_i^{+}\} ]
    =
    \frac{1}{2}\, (1-\lambda)\lambda^{|i|\wedge |i-1|},
    \qquad i\in\mathbb{Z}.
    \label{(6.13)}
\end{equation}

\noindent Here, $\mathfrak{F}$ has the distribution $\WSF_\lambda$, $T^{-}_{i-1}$ and $T^+_i$ are subtrees of $\mathbb{Z}^1$ with vertex sets $\{i-1, \, i-2,\, \ldots\}$ and $\{i,\, i+1, \, \ldots\}$, respectively.

\end{thm}

\begin{proof}

(i) The proof relies on the following general result (\cite[Theorem~9.4]{BI-LR-PY2001}): 
Let $G$ be a connected network, and let $\alpha(w_1, \ldots, w_k)$
denote the probability that $k$ independent RW's on the network started at $w_1$, $\ldots$, $w_k$ have no pairwise intersections. Then the number of trees in the wired uniform spanning forest is a.s.\
\begin{equation}
    \label{e:treenum}
    K=
    \sup\{ k:\ \exists w_1, \ldots, w_k,\ \alpha(w_1,\cdots,w_k)>0 \}.
\end{equation}

We first study the number of trees in the uniform spanning forest.

The case $d \ge 4$ is easy: According to Theorem \ref{T:intersection}, two independent $\RW_\lambda$'s on $\Z^d$ intersect finitely often a.s., so by
\eqref{e:treenum}, the number of trees in the uniform spanning forest associated with $\RW_{\lambda}$ on $\Z^d$ is a.s.\ infinite.

Consider now the case $d = 2$ or $3$. Let $(X_n^{(j)})_{n=0}^{\infty}$, $1 \le j \le 2^d$, be independent $\RW_\lambda$'s on $\Z^d$ starting at $o$. Note that the lower limit
$$
\liminf_{n \to \infty} \alpha (X^{(1)}_n, \ldots, X^{(2^d)}_n)
$$

\noindent is a.s.\ greater than or equal to the probability that $(X_n^{(j)})_{n=0}^{\infty}$, $1 \le j \le 2^d$, eventually direct into different orthants. The latter probability is strictly positive according to Theorem \ref{T:hkZd}(ii). Consequently, there exist $\varepsilon_0 >0$ and $n_0 \in \mathbb{N}$ such that
$$
\P \{ \alpha (X^{(1)}_{n_0}, \ldots, X^{(2^d)}_{n_0}) > \varepsilon_0 \} >0.
$$

\noindent A fortiori, there are $v_1$, $\ldots$, $v_{2^d}$ such that $\alpha(v_1, \ldots, v_{2^d}) > \varepsilon_0$. By \eqref{e:treenum}, there are at least $2^d$ trees in the uniform spanning forest associated with $\RW_{\lambda}$ on $\Z^d$.

To prove that the number of trees is at most $2^d$, let us consider $2^d + 1$ independent $\RW_\lambda$'s on $\Z^d$ starting at any initial points. Since there are $2^d$ orthants in $\Z^d$, Lemma \ref{P:axialplane} implies that a.s.\ there are at least two of them eventually directing into a common orthant. By Lemma \ref{L:intersect2}, these two $\RW_\lambda$'s intersect i.o. with probability $1$. Therefore,
 \[\sup\{k:\ \exists w_1,\cdots,w_k,\ \alpha(w_1,\cdots,w_k)>0\} \leq
   2^d. \]
 Therefore the number of trees in the uniform spanning forest is exactly $2^d$ by
\eqref{e:treenum}.

Now fix $d\geq 2$ and $\lambda\in (0, \, 1)$. Write
\begin{align*}
    |F|_{c_\lambda}
 &= \sum\limits_{e\in F}c_\lambda(e),\ F\subset E\left(\mathbb{Z}^d\right),
    \\
    |K|_\pi
  &= \sum\limits_{x\in K}
    \pi(x), \ K\subset\mathbb{Z}^d,
    \\
    \psi (\mathbb{Z}^d,\, t)
 &= \inf\left\{ |\partial_E K|_{c_\lambda}:\ t\leq\vert K\vert_\pi<\infty\right\},\ t>0;
\end{align*}

\noindent where $\partial_EK=\left\{\{x,y\}\in E\left(\mathbb{Z}^d\right):\ x\in K,y\notin K\right\}$, and $\pi(x) := \left(d_x^++d_x^-\lambda\right)\lambda^{-\vert x\vert}$, $x\in \mathbb{Z}^d$, is an invariant measure of the walk. 
Recall from Theorem~\ref{T:main1} that $\rho_\lambda=\frac{2\sqrt{\lambda}}{1+\lambda}<1$. By \cite[Theorem 6.7]{LR-PY2016},
$$\inf\left\{\frac{\vert\partial _EK\vert_{c_\lambda}}{\vert K\vert_\pi};\ \emptyset\not= K\subseteq\mathbb{Z}^d\ \mbox{is finite}\right\}\geq
  1-\rho_\lambda>0.$$
Thus for any $t>0,$ $\psi\left(\mathbb{Z}^d,t\right)\geq (1-\rho_\lambda)t.$ Since
$$\inf\limits_{x\in\mathbb{Z}^d}\left(d_x^++d_x^-\lambda\right)\lambda^{-\vert x\vert}>2d\lambda>0,$$
by \cite[Theorem 10.43]{LR-PY2016}, $\USF_\lambda$-a.s. every tree has only one end.


(ii) It remains to prove (\ref{(6.13)}). For any $n\in\mathbb{N},$ let $G_n=[-n,n]\cap\mathbb{Z}^1$ be the induced subgraph of tree $\mathbb{Z}^1,$ and $G_n^*$ the graph obtained from $G_n$ by identifying all vertices of $\mathbb{Z}^1\setminus G_n$ to a single vertex $z_n$ and deleting all the self-loops. 
Note $G_n^*$ is a simple cycle of length $2(n+1),$ and $z_n$ is adjacent to $n$ and $-n.$ Endow $G_n^*$ with the following edge conductance function $c_\lambda(\cdot):$
\begin{align*}
    c_\lambda(\{i-1,i\})
 &= \lambda^{-\left(\vert i\vert\wedge\vert i-1\vert\right)}, \qquad i\in [- (n-1), \, n] \cap \Z^1,
    \\
    c_\lambda(\{z_n,n\})
 &= c_\lambda(\{z_n,-n\})=\lambda^{-n}.
\end{align*}
Clearly all spanning trees of $G_n^*$ are of the form $G_n^*\setminus\{e\}$ for some edge $e$ of $G_n^*.$ Let
$$\Xi\left(G_n^*\setminus\{e\}\right)=\prod\limits_{f\in E(G_n^*)\setminus\{e\}}c_\lambda(f)=\frac{1}{c_\lambda(e)}\prod\limits_{f\in E(G_n^*)}c_\lambda(f).$$
By the definition of $\WSF$, for any $i\in\mathbb{Z}^1,$
\begin{align*}
    \WSF_\lambda\left[\mathfrak{F}= \left\{T_{i-1}^{-},T_i^{+}\right\}\right]
 &= \lim\limits_{n\rightarrow\infty}\frac{\Xi\left(G_n^*\setminus\{i-1,i\}\right)}
       {\sum\limits_{e\in E(G_n^*)}\Xi\left(G_n^*\setminus\{e\}\right)}
    =\lim\limits_{n\rightarrow\infty}\frac{c_\lambda\left(\{i-1,i\}\right)^{-1}}
          {\sum\limits_{e\in E(G_n^*)}c_\lambda\left(\{e\}\right)^{-1}}
    \\
 &= \lim\limits_{n\rightarrow\infty}\frac{\lambda^{\vert i\vert\wedge\vert i-1\vert}}
    {2\sum\limits_{k=0}^n\lambda^k}=\frac{1}{2}(1-\lambda)\lambda^{\vert i\vert\wedge\vert i-1\vert} ,
\end{align*}
\noindent as desired.
\end{proof}

\subsection{Discussions of the singularity problem}
Recall that both $\FSF$ and $\WSF$ are determinantal point processes (DPPs) on the set of all edges of a graph (\cite{LR2003}). For the singularity problem,
the following general version is false: Given an infinite countable set $E$ and any two closed subspaces $H_1$ and $H_2$ of $\ell^2(E)$ with $H_1\subsetneq H_2,$ the distributions of DPPs corresponding to $H_1$ and $H_2$ are mutually singular. See \cite{LR2003} p.~203.

\begin{thm}\label{thm6.3}
  Let $G$ be any graph whose simple cycles are of uniformly bounded lengths. For any network on $G$ with positive conductances,
the corresponding $\FSF$ has only one tree.



\end{thm}

\noindent {\it Proof.} 
Let $G=(V,E)$ be any graph whose simple cycles are of uniformly bounded lengths. Given any positive conductance function $c(\cdot)$ on $E$.


Consider the exhaustion $G_n=B_G(n)$, $n\in\mathbb{N}$, of $G$. Let $\ell\in\mathbb{N}$ be the maximal length of all simple cycles of $G.$
Given any two distinct vertices $x,y\in V.$ Choose $n_0\in\mathbb{N}$ such that $x,y\in G_n,\ \forall n\geq n_0;$ and let $d_G(x,y)$
be the graph distance between $x$ and $y$ in $G.$ Then for any $n\geq n_0$ and any spanning tree $T_n$ of $G_n,$ the distance $d_{T_n}(x,y)$ between
$x$ and $y$ in $T_n$ is at most $\ell d_G(x,y).$

{Indeed}, suppose conversely
$$d_{T_n}(x,y)\geq \ell d_G(x,y)+1,$$
and let $\gamma_1=x_0x_1\cdots x_{n_1}$ (resp. $\gamma_2=y_0y_1\cdots y_{d_G(x,y)}$) be the geodesic from $x$ to $y$ in $T_n$ (resp. $G$).
Here
$$x_0=y_0=x,\ x_{n_1}=y_{d_G(x,y)}=y,\ n_1=d_{T_n}(x,y).$$
Assume successive intersection points of $\gamma_1$ and $\gamma_2$ are
$$x_{i_0}=y_{j_0},\ x_{i_1}=y_{j_1},\ \cdots,\ x_{i_k}=y_{j_k},$$
where $i_0=0<i_1<\cdots<i_k=n_1,\ j_0=0<j_1<\cdots<j_k=d_G(x,y)$ and $1\leq k\leq d_G(x,y).$ Note for each $0\leq r\leq k-1,$
the segments of $\gamma_1$ and $\gamma_2$
between $x_{i_r}=y_{j_r}$ and $x_{i_{r+1}}=y_{j_{r+1}}$ forms a simple cycle $C_r$ in $G;$ and the total length of all these $C_r$s
is $n_1+d_G(x,y).$ Hence there is a simple cycle $C_r$ in $G$ whose length is at least
$$\frac{n_1+d_G(x,y)}{k}\geq\frac{\ell d_G(x,y)+1+d_G(x,y)}{d_G(x,y)}>\ell+1;$$
which is a contradiction to the definition of $\ell.$

Hence for any $n\geq n_0,$ $\mu_n^F$-a.s. $d_{T_n}(x,y)\leq \ell d_G(x,y),$ where $T_n$ has the law $\mu_n^F$. Taking limit $n\rightarrow\infty$, 
we have that
$$\FSF\text{-a.s.}\ T,\ d_T(x,y)\leq \ell d_G(x,y)<\infty,$$
where $T$ obeys the law $\FSF$, $d_T(x,y)$ is the graph distance between $x$ and $y$ in $T.$ This means that any two distinct vertices $x$ and
$y$ in $G$ is connected in $T$ for $\FSF$-a.s.\ $T$. Therefore, $T$ is $\FSF$-almost surely a tree, namely $\FSF$ has only one tree.
\qed\\



\begin{remark}
  Let $(G,\, c)$ be a network such that there exists a number $n_0$ with the property that at least two disjoint components of $G \setminus B_G(n_0)$ are transient (i.e., the associated random walks are transient). By \eqref{e:treenum}, the $\WSF$ a.s. has at least two trees. If all the simple cycles have uniformly bounded lengths, then we have from Theorem~\ref{thm6.3} that the $\FSF$ and $\WSF$ are mutually singular. 
\end{remark}

\flushleft{Zhan Shi\\
LPMA, Universit\'{e} Paris VI\\
4 place Jussieu, F-75252 Paris Cedex 05\\
France\\
E-mail: \texttt{zhan.shi@upmc.fr}}\\

\flushleft{Vladas Sidoravicius\\
NYU-ECNU Institute of Mathematical Sciences at NYU Shanghai\\
\& Courant Institute of Mathematical Sciences\\
New York, NY 10012, USA\\
E-mail: \texttt{vs1138@nyu.edu}}\\

\flushleft{He Song\\
  Department of Mathematical Science, Taizhou University\\
  Taizhou 225300, P. R. China\\
  Email: \texttt{tayunzhuiyue@126.com}}

\flushleft{Longmin Wang and Kainan Xiang\\
School of Mathematical Sciences, LPMC, Nankai University\\
Tianjin 300071, P. R. China\\
E-mails: \texttt{wanglm@nankai.edu.cn} (Wang)\\
\hskip 1.4cm \texttt{kainanxiang@nankai.edu.cn} (Xiang)}
\end{document}